\documentclass[]{amsart}
\pdfoutput=1

\usepackage[utf8]{inputenc}
\usepackage[T1]{fontenc}
\usepackage{hyperref}
\usepackage{amsfonts,amssymb,amsthm,amsmath}
\usepackage{algorithm}
\usepackage{algpseudocode}
\usepackage{graphbox,graphicx}
\usepackage{url}
\usepackage{mleftright}
\usepackage{bbm}
\usepackage{comment}
\usepackage{xcolor}
\usepackage[normalem]{ulem}

\usepackage[backend=biber,maxbibnames=9,style=numeric-comp]{biblatex}
\calclayout

\graphicspath{{figs/}}

\newtheorem{theorem}{Theorem}
\newtheorem{proposition}[theorem]{Proposition}

\newtheorem{lemma}[theorem]{Lemma}
\newtheorem{definition}[theorem]{Definition}
\newtheorem{corollary}[theorem]{Corollary}

\newtheorem{remark}[theorem]{Remark}
\newtheorem*{remark*}{Remark}
\numberwithin{theorem}{section}

\makeatletter
\newcommand{\mylabel}[2]{#2\def\@currentlabel{#2}\label{#1}}
\makeatother

\newcommand{\HH}{\mathcal{H}}
\newcommand{\VV}{\mathcal{V}}
\newcommand{\FF}{\mathcal{F}}
\newcommand{\Ret}{\operatorname{Ret}}
\newcommand{\Exp}{\operatorname{Exp}}

\graphicspath{{figures/}}

\makeatletter
\@namedef{subjclassname@2020}{%
	\textup{2020} Mathematics Subject Classification}
\makeatother

\title[Sub-Riemannian Random Walks]{Sub-Riemannian Random Walks --\\ From  Connections to Retractions}

\author{Michael Herrmann}
\address{Institute for Partial Differential Equations \\Technische Universit{\"a}t Braunschweig \\
	Universit{\"a}splatz 2, 38106 Braunschweig \\
	Germany}
\email{michael.herrmann@tu-braunschweig.de}

\author{Pit Neumann}
\address{Institute for Numerical and Applied Mathematics\\
	University of G{\"o}ttingen \\
	Lotzestra{\ss}e 16-18 \\ 37083 G{\"o}ttingen\\
	Germany}
\email{pit.neumann@stud.uni-goettingen.de}

\author{Simon Schwarz}
\address{Institute for Mathematical Stochastics\\ 
	University of G{\"o}tttingen\\
	Goldschmidtstra{\ss}e 7\\ 37077 G{\"o}ttingen\\
	Germany}
\email{simon.schwarz@uni-goettingen.de}

\author{Anja Sturm}
\address{Institute for Mathematical Stochastics\\ 
	University of G{\"o}ttingen\\
	Goldschmidtstra{\ss}e 7\\ 37077 G{\"o}ttingen\\
	Germany}
\email{anja.sturm@mathematik.uni-goettingen.de}

\author{Max Wardetzky}
\address{Institute for Numerical and Applied Mathematics\\
	University of G{\"o}ttingen \\
	Lotzestra{\ss}e 16-18 \\ 37083 G{\"o}ttingen\\
	Germany}
\email{wardetzky@math.uni-goettingen.de}

\thanks{The authors acknowledge financial support by the Deutsche Forschungsgemeinschaft (DFG, German Research Foundation) via project A02 of the Collaborative Research Center 1456.}

\begin{document}
	
	
	
\begin{abstract}
We study random walks on sub-Riemannian manifolds using the framework of retractions, i.e., approximations of normal geodesics.  We show that such walks converge to the correct horizontal Brownian motion if normal geodesics are approximated to at least second order. In particular, we (i) provide conditions for convergence of geodesic random walks defined with respect to normal, compatible, and partial connections and (ii) provide examples of computationally efficient retractions, e.g., for simulating anisotropic Brownian motion on Riemannian~manifolds. 
\end{abstract}
\vspace*{-1cm}
\maketitle
\vspace*{-2ex}
	
\section{Introduction}\label{sec:intro}
Given a smooth, compact, $n$-dimensional Riemannian manifold $(M,g)$ without boundary, consider the following simple algorithm. Pick an initial position $x\in M$, fix $\varepsilon>0$, and keep repeating the following two steps of sampling (S) and walking (W):

\begin{enumerate}
\item[(S)] Sample a unit tangent vector $u \in \mathbb{S}_{x}:=\{u\in T_{x}M : \lVert u\rVert_g = 1\}$ \emph{uniformly} with respect to the Riemannian metric $g$.
\item[(W)] Follow the unique geodesic with initial conditions $(x,u)$ for time $\varepsilon \sqrt{n}$, and update $x$ to be the endpoint of this geodesic. 
\end{enumerate}
A famous result by J{\o}rgensen~\cite{jorgensen1975}  states that the resulting \emph{geodesic random walk} converges to the Brownian motion on $(M,g)$ as $\varepsilon\to0$.

Does a similar result hold in the \emph{sub-Riemannian} case? Recall that
a sub-Riemannian structure $(\HH, h)$ of a smooth  $n$-dimensional manifold $M$ consists of a smooth linear subbundle $\HH\subset TM$ of constant rank $k\leq n$ and a smoothly varying positive definite quadratic form  $h$ on $\HH$. Notice that $h(X,Y)$ is defined for \emph{horizontal vector fields} $X,Y\in \Gamma(\HH)$ only; there is no metric information for vectors transversal to $\HH$ if $k<n$.

Trying to mimic the above algorithm in the sub-Riemannian case poses (at least) two challenges for walking and sampling: (i) With respect to walking, sub-Riemannian geodesics starting at a point $x$ are in general not uniquely determined by an initial (horizontal) vector $u \in \HH_{x}$. Rather, so-called \emph{normal geodesics} are determined by a covector $\alpha \in T^{*}_{x}M$. 
(ii) With respect to sampling, while it is certainly possible to uniformly sample a  (horizontal) unit vector $u \in \HH_{x}$ with respect to the sub-Riemannian metric $h$, there is no canonical way of constructing a unique covector $\alpha \in T^{*}_{x}M$ from $u$ (required for walking along normal geodesics).

Resolving these challenges in order to obtain \emph{horizontal geodesic random walks} requires to make an additional choice. One possible  choice concerns a \emph{complement} $\VV$ of $\HH$, i.e., a smooth linear subbundle $\VV\subset TM$ such that $\VV\cap \HH = \{0\}$ and $TM = \HH \oplus \VV$. While there exist many choices of complements $\VV$, some sub-Riemannian structures admit natural choices, such as the Reeb vector field for contact structures or the vertical bundle for frame bundles. When not stated otherwise, we will assume that some choice of $\VV$ is given.

Under such a choice of $\VV$, one can adapt the above algorithm to the sub-Riemannian case as follows. Pick an initial position $x\in M$, fix $\varepsilon>0$, and keep repeating the following two steps:

\begin{enumerate}
\item[\mylabel{sR-sample}{($\tilde{\mathrm{S}}$)}] Sample a unit horizontal vector $u \in \mathbb{S}_{x}:=\{u\in \HH_{x} : \lVert u\rVert_h = 1\}$ \emph{uniformly} with respect to the sub-Riemannian metric $h$. \item[\mylabel{sR-walk}{($\tilde{\mathrm{W}}$)}] 
Let $\alpha_{u} \in T^{*}_{x}M$  be the unique covector satisfying  $\alpha_{u}(\tilde u) = h(u,\tilde u)$ for all $\tilde u \in \HH_{x}$ and $\alpha_{u}(v) = 0$ for all $v\in \VV_{x}$. Follow the unique normal geodesic with initial conditions $(x,\alpha_{u})$ for time $\varepsilon \sqrt{k}$, and update $x$ to be the endpoint of this geodesic. 
\end{enumerate}

For bracket-generating sub-Riemannian structures, it was shown in \cite{BOSCAIN2017} (see also \cite{agrachev2018,gordina2017convergence}) that the horizontal geodesic walk arising from~\ref{sR-sample} and~\ref{sR-walk} indeed converges to a respective \emph{horizontal Brownian motion generated by a $\VV$-dependent} \emph{sub-Laplacian}.
One could let the story end with this convergence result. We decided to let our story begin here.

The convergence results in \cite{BOSCAIN2017} pertain to random walks along \emph{normal geodesics}. We are interested in an extension of this picture to random walks that possibly arise from different setups.
Our motivation for investigating such an extension is twofold. On the one hand, normal geodesics arise naturally in the sub-Riemannian setting. However, geodesics arising from \emph{normal}, \emph{compatible}, or \emph{partial} connections (all of which are natural connections associated with sub-Riemannian structures; see, e.g.,~\cite{CHENG2021112387,Molina-Grong-2017,Grong2020AffineCA,,Grong2022-notes,Grong-Thalmaier2019,junne2023invariance,LANGEROCK2003203}) do \emph{not} in general agree with normal geodesics. The fact that different affine connections can lead to different geodesics, may result in different
limiting processes for the resulting geodesic random walks. We give conditions under which these processes agree with the correct horizontal Brownian motion. Notice that horizontal Brownian motion is generated by an attendant sub-Laplacian, and thus depends on the choices of $\HH$, $h$, and $\VV$.

Another source of motivation for asking for alternatives of random walks along normal geodesics arises from a computational perspective: Computing geodesics is computationally costly in general. We ask for walks that are computationally efficient, while still providing the correct 
limiting processes, thereby extending the results of \cite{schwarz2023efficient} from the Riemannian to the sub-Riemannian setting.

We combine these two strands of motivation under the roof of \emph{retraction-based} random walks. We show that if normal geodesics are approximated to at least second order -- i.e., by what we call \emph{second-order retractions} --, then the \emph{generators} of the resulting stochastic processes converge to the correct (horizontal) sub-Laplacian, i.e., to the generator of horizontal Brownian motion; see Theorem~\ref{thm:convergence-generators-ret2}.  As a consequence we infer that if the sub-Riemannian structure is bracket-generating, then the stochastic processes themselves converge in distribution to horizontal Brownian motion; see Theorem~\ref{thm:convergence-processes-ret2}. Notice that without this condition, convergence of generators might not imply convergence of the resulting processes.

Our exposition is structured as follows. First, we recall basic notions of sub-Riemannian geometry in Section~\ref{sec:basics-sub-Riemannian}. We then introduce retraction-based random walks and discuss their convergence to the correct horizontal Brownian motion in Section~\ref{sec:rb-rw}. In Section \ref{sec:normal-compatible} we discuss the relationship between normal, compatible and partial connections, and we show that (certain) compatible connections indeed yield second-order retractions. As an example we consider a certain compatible connection on the frame bundle over a smooth manifold  $M$ with given affine connection in Section~\ref{sec:frame-bundles}. Using the frame bundle, one can generalize the Eells--Elworthy--Malliavin construction of Brownian motion to \emph{anisotropic} Brownian motion, see, e.g.,~\cite{grong2022,Sommer-Svane2017}.
Finally, in Section ~\ref{sec:algo}, we discuss second-order retractions that yield computationally efficient algorithms.
	
	\section{Basic Notions for sub-Riemannian Structures}\label{sec:sR-basics}\label{sec:basics-sub-Riemannian}
	Throughout our exposition we let $M$ denote a smooth and connected manifold without boundary. We will often additionally assume that $M$ is compact. A \emph{sub-Riemannian structure} $(\HH, h)$ on $M$ consists of a smooth linear subbundle $\HH\subset TM$ of constant rank $k\leq n$, and a smoothly varying positive definite quadratic form  $h$ on $\HH$. 
	 Sections $X\in \Gamma(\HH)$ are called \emph{horizontal} vector fields. Notice that $h(X,Y)$ is defined only for horizontal vector fields; there is no metric information for vectors transversal to $\HH$ if $k<n$. Smooth curves $\gamma:[a,b]\to M$ with $\dot\gamma(t) \in \HH$ for all $t$ are called \emph{horizontal curves}. The \emph{length} of a horizontal curve is defined by
	$$
	\mathrm{len}(\gamma) := \int_{a}^{b}h(\dot \gamma, \dot \gamma)^{\frac 12}\, dt \ .
	$$
	The infimum over the lengths over all (piecewise) smooth horizontal curves connecting two points in $M$ gives rise to the Carnot--Carath\'eodory distance $d_{\HH}$. Notice that, in general, this distance can be infinite: E.g., according to Frobenius' theorem, if the commutator $[X,Y]$ of any pair of horizontal vector fields is again horizontal, then $\HH$ is integrable, i.e., $\HH$ foliates $M$ with $k$-dimensional leaves. In this situation,  points lying in different leaves have infinite distance since horizontal curves are bound to stay within single leaves. On the other end of the spectrum stands the assumption that $\HH$ 
	is \emph{bracket-generating} (also known as $\HH$ satisfying H\"ormander's condition), meaning that the horizontal vector fields along with all of their (nested) commutators span all of $TM$. Then the Chow--Rashevskii theorem asserts that $d_{\HH}$ is indeed finite and thus turns $M$ into a metric space. In other words, if H\"ormander's condition is satisfied, then any two points can be connected by a horizontal curve. 
For a comprehensive overview of sub-Riemannian structures, we refer to the surveys~\cite{agrachev_barilari_boscain_2019,montgomery2002tour,Strichartz}.

\subsection{Normal geodesics}\label{sec:normal-geodesics}
As in the Riemannian setting, minimizers of the length functional among all (sufficiently regular) curves connecting two endpoints $x, y \in M$ are called (horizontal) \emph{geodesics}. Different from the Riemannian setting, though, sub-Riemannian geodesics are in general not uniquely  determined by an initial point $x\in M$ and an initial horizontal velocity $u \in \HH_{x} \subset T_{x}M$. This can be seen by a simple dimension counting argument considering the Chow--Rashevskii theorem. Nonetheless, so-called \emph{normal} sub-Riemannian geodesics are uniquely determined by an initial point $x\in M$  and an initial \emph{covector} $\alpha \in T^{*}_{x}M$. We briefly review the corresponding construction. A sub-Riemannian structure gives rise to a smooth linear mapping 
	\begin{align*}
	\sharp: \Gamma(T^{*}M) &\to \Gamma(\HH) \ , \quad \alpha \mapsto \alpha^{\sharp} \ , \quad \text{where $\alpha^{\sharp}$ satisfies}\\
	h(\alpha^{\sharp} , X) &= \alpha(X)  \quad \text{for all $X \in \Gamma(\HH)$} \ .
	\end{align*}
	Notice that if $\HH \neq TM$, then the operator ${\sharp}$ has a non-trivial kernel. Nonetheless, this operator gives rise to a (degenerate if $\HH \neq TM$) \emph{cometric} $g$ on $T^{*}M$ via
	$$
	g(\alpha, \beta) := h(\alpha^{\sharp}, \beta^{\sharp}) \ .
	$$
	This cometric, in turn, yields a \emph{Hamiltonian} 
	$$
	H: T^{*}M \to \mathbb{R} \ , \quad H(\alpha) := \frac 12 g(\alpha,\alpha) \ .
	$$
	The canonical symplectic form on $T^{*}M$ then yields the corresponding Hamiltonian vector field $X_{H} \in \Gamma (T \,T^{*}M)$, and projections of the flow lines of $X_{H}$ back to $M$ give rise to normal geodesics:
\begin{definition}[Normal geodesics]
	A curve $\gamma:[a,b]\to M$ is called a \emph{normal} sub-Riemannian geodesic if and only if there exists a curve $\alpha :[a,b]\to T^{*}M$ with $\dot\alpha(t) = X_{H}|_{\alpha(t)}$ and $\gamma = \pi_{M}\circ \alpha$, where $\pi_{M}:T^{*}M \to M$ is the bundle projection.
\end{definition}
	In canonical coordinates the curve $\alpha(t) = (x^{i}(t), p_{i}(t))$ appearing in the previous definition satisfies the following system of Hamiltonian differential equations:
	\begin{align}\label{eq:normal-geod-Hamiltonian-1}
	\begin{split}
	\dot x^{i}(t) = \frac{\partial H}{\partial p_{i}}(t) &\iff \dot x^{i}(t)= g^{ij}(x(t))\,p_{j}(t) \ , \\
	\quad \dot p_{i}(t) = - \frac{\partial H}{\partial x^{i}}(t) &\iff\dot p_{i}(t)= -\Gamma^{jk}_{i}(x(t))\, p_{j}(t)p_{k}(t) \ , 
	\end{split}
	\end{align}
	where we have used the Einstein summation convention and one defines 
	\begin{align}\label{eq:normal-geod-Hamiltonian-2}
	g^{ij}:= g(dx^{i}, dx^{j}) \quad \text{and} \quad \Gamma^{ij}_{k}:= \frac 12 \frac{\partial g^{ij}}{\partial x^{k}} \ .
	\end{align}
	Notice that the first equation in the above Hamiltonian system  is equivalent to $\dot\gamma(t) = \alpha^{\sharp}(t)$ for the curve $\gamma = \pi_{M}\circ \alpha$; thus $\gamma$ is indeed horizontal. By construction, normal geodesics are determined by initial data $(x_{0}, \alpha_{0})$ with $x_{0}\in M$ and $\alpha_{0} \in T^{*}_{x_{0}}M$. 
\begin{remark}	
While all normal geodesics are locally length minimizing, not all length minimizing horizontal curves are necessarily normal geodesics in general. For our purposes, however, it will suffice to consider normal geodesics. 
\end{remark}

As mentioned before, the choice of a \emph{complement} $\VV$ for $\HH$, i.e., a smooth linear subbundle $\VV\subset TM$ such that $\VV\cap \HH = \{0\}$ and $TM = \HH \oplus \VV$, will be central for our exposition. In the sequel we assume that some choice of a complement $\VV$ has been made. We record the following definition:

\begin{definition}[Normal geodesics with horizontal initial conditions]\label{def:hor-init-cond}
		Let $(\HH,h)$ be a sub-Rie\-mann\-ian structure with complement $\VV$. We say that a normal geodesic $\gamma: [0,\varepsilon) \to M$ has \emph{horizontal} initial conditions $(x,u)$, with $x \in M$ and $u \in \HH_{x}$, if $\gamma$ is the (unique) normal geodesic with initial conditions $x\in M$ and $\alpha_{u} \in T_{x}^{*}M$, where  $\alpha_{u}$ is defined by $\alpha_{u}^{\sharp} = u$ and $\alpha_{u}|_{\VV} = 0$.
	\end{definition}

	\subsection{Horizontal Lie derivatives, sub-Laplacians, and Brownian motion} 
A sub-Riemannian structure $(\HH,h)$  together with a complement $\VV$ yield the notions of horizontal Lie derivatives, horizontal  divergence, sub-Laplacians, and horizontal  Brownian motion:

	\begin{definition}[Horizontal Lie derivative]\label{def:hor-Lie}
		Let $(\HH,h)$ be a sub-Riemannian structure with complement $\VV$. Let $A$ be a $(0,m)$-tensor on $M$, let $X_{1}, \dots, X_{m} \in \Gamma(\HH)$ be horizontal vector fields, and let $Y\in \Gamma(TM)$ be an arbitrary vector field. Define the \emph{horizontal Lie derivative} of $A$ by
		\begin{align}\label{eq:hor-Lie}
		\left(\mathcal{L}_{Y}^{\VV}A\right)(X_{1}, \dots, X_{m}):= Y(A(X_{1}, \dots, X_{m})) - \sum_{i=1}^{m} A(X_{1}, \dots, \mathrm{pr}_{\HH}[Y, X_{i}], \dots, X_{m}) \ ,
		\end{align}
		where $\mathrm{pr}_{\HH}$ denotes the projection to $\HH$ (which depends on the choice of $\VV$). Notice that this definition makes sense also for ``horizontal'' $(0,m)$-tensors, i.e., those $A$ for which $A(X_{1}, \dots, X_{m})$ is \emph{only} known for horizontal $X_{i}$. 
	\end{definition}
	
	Using the notion of horizontal Lie derivatives, one can define horizontal divergence and sub-Laplacians. In order to define the former, let $(\HH,h)$ be a sub-Riemannian structure, and let $\omega_{\HH}$ be a local ``horizontal volume form'' on $\HH$, i.e., a horizontal $k$-form that satisfies 
	\begin{align}\label{eq:hor-volume-form}
	\omega_{\HH}(E_{1}, \dots, E_{k}) \in \{-1,1\}
	\end{align}
	for all (local) orthonormal horizontal frames. (Notice that as long as the value of $\omega_{\HH}(E_{1}, \dots, E_{k})$ is locally constant, we do not care about the sign since we do not assume any notion of orientation of $\HH$. Notice also that $\omega_{\HH}$ is unique up to sign.) Following~\cite{BOSCAIN2017}, we then make the following definitions:
	
	\begin{definition}[Horizontal divergence and sub-Laplacian]\label{def:hor-div-and-sub-Lap-hor}
		Let $(\HH,h)$ be a sub-Riemannian structure with complement $\VV$, let $X \in \Gamma(TM)$, and let $\omega_{\HH}$ be defined by~\eqref{eq:hor-volume-form}. Then the \emph{horizontal divergence} of $X$ is the function $\mathrm{div}^{\VV}(X)$ defined by
		\begin{align}\label{eq:hor-div}
		\mathcal{L}_{X}^{\VV}\omega_{\HH} = \mathrm{div}^{\VV}(X) \, \omega_{\HH} \ .
		\end{align}
		Moreover, define the \emph{sub-Laplacian} by 
		\begin{align}\label{eq:hor-Lap}
		\Delta^{\VV} f := \mathrm{div}^{\VV}((df)^{\sharp}) \ .
		\end{align}
	\end{definition}
	
If the sub-Riemannian structure $(\HH, h)$ is bracket-generating, every sub-Laplacian of the form $\Delta^{\VV}$ generates a horizontal diffusion process $X_{t}^{\VV}$, see \cite{BOSCAIN2017,habermann2017,Strichartz}.
This diffusion corresponds to (horizontal) \emph{Brownian motion}, which not only depends on $(\HH, h)$ but also on the choice of a complement $\VV$, just like the sub-Laplacian $\Delta^{\VV}$ itself.
Sub-Laplacians play a central role in the next section in the context of random walks. 
	
\section{Retraction-based Random Walks}\label{sec:rb-rw}

Recall the (horizontal) geodesic random walk resulting from repeating Steps \ref{sR-sample} and \ref{sR-walk} in Section~\ref{sec:intro}. It was shown in~\cite{BOSCAIN2017} that these walks satisfy a functional central limit theorem, i.e., converge to the horizontal Brownian motion corresponding to the sub-Laplacian $\Delta^{\VV}$ as $\varepsilon \to 0$. In this section we extend the concept of (horizontal) geodesic random walks to \emph{retraction-based} random walks. Retractions were originally introduced within the context of optimization on Riemannian manifolds as efficient approximations of the exponential map, see, e.g., \cite{absil2008,absil2012,adler2002,bonnabel2013}.
	Extending retractions to sub-Riemannian manifolds offers a natural framework for examining the limit of geodesic random walks constructed with respect to different connections.
	
\begin{definition}[Horizontal retractions]\label{def:sub-Riemannian-retraction}
Let $(\HH, h)$ be a sub-Riemannian structure on $M$. A \emph{horizontal retraction} is a smooth map $\Ret: \HH \to M$ that for all $x\in M$ and all $u\in \HH_{x}$ satisfies
		\begin{align*}
		\Ret_x(0) = x \quad \text{and} \quad \frac{d}{dt}\Ret_x (t u) \big\lvert_{t=0} = u  \ ,
		\end{align*}
		where $\Ret_x$  denotes the restriction to $\HH_x$. 
\end{definition}

Retraction can be used to extend the notion of geodesic random walks:

\begin{definition}[Retraction-based random walk]\label{def:ret-rand-walk}
Let $(\HH, h)$ be a sub-Riemannian structure on $M$ with horizontal retraction $\Ret$. A \emph{retraction-based random walk} refers to the following algorithm. Pick an initial position $x\in M$, fix $\varepsilon>0$, and keep repeating the following two steps:
	
\begin{enumerate}
\item[\mylabel{Ret-sample}{($\tilde{\mathrm{S}}$)}] Sample a unit horizontal vector $u \in \mathbb{S}_{x}:=\{u\in \HH_{x} : \lVert u\rVert_h = 1\}$ \emph{uniformly} with respect to the sub-Riemannian metric $h$.
\item[\mylabel{Ret-walk}{($\tilde{\mathrm{R}}$)}] Update $x$ according to $x\leftarrow \Ret_{x}(\varepsilon \sqrt{k} u)$. 
\end{enumerate}
\end{definition}

Notice that (i) the sampling step \ref{Ret-sample} is the same as the sampling step for (horizontal) geodesic random walks and that (ii) the walking step \ref{Ret-walk} is the same as the walking step \ref{sR-walk} when  $\Ret$ is replaced by $\Exp$, where the horizontal exponential map  $\Exp_{x}(tu)$ is defined by following the (unique) normal geodesic with horizontal initial conditions $(x,u)$ for time $t$ (see Definition~\ref{def:hor-init-cond}). Notice furthermore that (thus far) the choice of a complement $\VV$ is absent from our definition of retraction-based random walks. Therefore, quite evidently, one cannot expect that \emph{every} retraction-based random walk converges to the correct horizontal Brownian motion as $\varepsilon\to 0$. In order to fix this, we relate retraction-based random walks to normal geodesics:
	
	\begin{definition}[Second-order horizontal retractions]\label{def:second-order-agreement}
		Let $\gamma, \tilde \gamma: (-\delta, \delta) \to M$ be two smooth curves in $M$ with $\gamma(0)= \tilde\gamma(0)$ and $\dot\gamma(0)= \dot{\tilde\gamma}(0)$. Let $x(t)$ and $\tilde x(t)$ be the coordinate representation of $\gamma$ and $\tilde\gamma$, respectively,  in some local coordinate chart around $\gamma(0)$. If 
		$$
		\frac{d^{2}}{dt^{2}} x(t) \big\lvert_{t=0} = \frac{d^{2}}{dt^{2}} \tilde x(t) \big\lvert_{t=0}  \ ,
		$$
		then we say that $\gamma$ and $\tilde \gamma$ \emph{agree up to second order} at $t=0$.
		
		Given a complement $\VV$ of $\HH$, we say that a retraction
		$\Ret$ is a \emph{second-order} horizontal retraction if the curve $t \mapsto \Ret_x (t u)$ agrees with the (unique) normal geodesic $\gamma$ with horizontal initial conditions $(x,u)$ (see Definition~\ref{def:hor-init-cond}) up to second order at $t=0$ for all $x \in M$ and all $u\in \HH_{x}$.
	\end{definition}
	
	\begin{remark}\label{rem:second-order-agreement-connection}
		It is easy to verify that two curves $\gamma$ and $\tilde \gamma$ with $\gamma(0)= \tilde\gamma(0)$ and $\dot\gamma(0)= \dot{\tilde\gamma}(0)$ agree up to second order at $t=0$ if and only if $\nabla_{\dot{\gamma} }\dot{\gamma}\rvert_{t=0} = \nabla_{\dot{\tilde\gamma} }\dot{\tilde\gamma}\rvert_{t=0}$ for one (and hence every) affine connection $\nabla$ on $M$. This fact will be important in the next section.
	\end{remark} 

Using second-order retractions, we now turn to the question of convergence of retraction-based random walks to horizontal Brownian motion.	
	We first deal with convergence of the infinitesimal generators and in a second step with convergence of the attendant stochastic processes. 
	For the former, consider the following operator: Given a sub-Riemannian structure $(\HH, h)$  and a (not necessarily second-order) horizontal retraction $\Ret:\HH\to M$, let
	\begin{equation*}
	L^{\varepsilon}_{\Ret} (f) := \frac{1}{\varepsilon^{2}} \left ( U^{\varepsilon}_{\Ret} (f) - f \right)
	\end{equation*}
	for any $f\in C^0 (M;\mathbb{R})$, where
	\begin{equation*}
	\left(U^{\varepsilon}_{\Ret} (f)\right) (x) := \frac{1}{\omega_{k}}\int_{\mathbb{S}_{x}} f\left (\Ret_{x}\left(\sqrt{k}\, \varepsilon u\right)\right) d u 
	\end{equation*}
	is the transition operator for the walking Step \ref{Ret-walk} in Definition~\ref{def:ret-rand-walk} and $\omega_{k}$ the volume of the unit sphere $\mathbb{S}_{x}\subset \HH_{x}$.
	We then have the following result:
	
	\begin{theorem}\label{thm:convergence-generators-ret2}
Let $M$ be compact, let $(\HH,h)$ be a sub-Riemannian structure on $M$ with complement $\VV$, and let $f \in C^{\infty}(M;\mathbb{R})$. Consider a second-order horizontal retraction $\Ret: \HH \to M$. Then 
		\begin{align}\label{eq:conv-generators-ret}
		\lim_{\varepsilon \to 0}L^{\varepsilon}_{\Ret} (f)  = \Delta^{\VV}f \ ,
		\end{align}
uniformly on $M$, where $\Delta^{\VV}$ denotes the horizontal Laplacian from~\eqref{eq:hor-Lap}. 
	\end{theorem}
	
	\begin{proof}
		The proof largely hinges on the corresponding result from~\cite{BOSCAIN2017} for \emph{geodesic} random walks.  Indeed, our definition of the sub-Laplacian defined by~\eqref{eq:hor-Lap} is identical to the \emph{microscopic Laplacian} with complement $\VV$ considered in Proposition 55 in~\cite{BOSCAIN2017}. 
		
		With these preliminaries at hand, Theorem 51 in~\cite{BOSCAIN2017} assures the validity of the statement of our theorem when $\Ret$ is replaced by $\Exp$, where the horizontal exponential map  $\Exp_{x}(tu)$ is defined by following the (unique) normal geodesic with horizontal initial conditions $(x,u)$ for time $t$. Then the statement of Theorem~\ref{thm:convergence-generators-ret2} immediately follows from a straightforward Taylor expansion of $\Ret_{x}(tu)$ and $\Exp_{x}(tu)$, while noticing that, by definition of second-order retractions,  these two maps agree up to second order at $t=0$ in the sense of Definition~\ref{def:second-order-agreement}. (See \cite{schwarz2023efficient} for a similar argument in the Riemannian case.)
	\end{proof}
	
	Notice that we did not have to assume the sub-Riemannian structure to be bracket-generating (i.e., to satisfy H\"ormander's condition) in the previous theorem. However, for convergence of the stochastic processes
	\begin{equation}\label{eq:def-stoch-proc}
	X^\varepsilon_t := x^\varepsilon_{\lfloor t/\varepsilon^2 \rfloor}\ , \quad t>0
	\end{equation}
resulting from Definition~\ref{def:ret-rand-walk} (with the usual parabolic scaling in time) we indeed require H\"ormander's condition. H\"ormander's condition implies that the topologies induced by the Carnot--Carath\'eodory metric and the smooth structure on $M$ coincide, see Theorem~2.3 in \cite{montgomery2002tour}. Under this condition we obtain the following result as a direct consequence of the convergence in Equation~\eqref{eq:conv-generators-ret}, see Theorem 69 in \cite{BOSCAIN2017}: 
	
	\begin{theorem}\label{thm:convergence-processes-ret2}
Let $M$ be compact, and let $(\HH,h)$ be a \emph{bracket-generating} sub-Riemannian structure  on $M$ with complement $\VV$. Consider a second-order horizontal retraction $\Ret: \HH \to M$ and the resulting stochastic process $(X_t^\varepsilon)_{t>0}$ defined in Equation \eqref{eq:def-stoch-proc}. Then $(X_t^\varepsilon)_{t>0}$ converges in distribution to the horizontal Brownian motion with generator $\Delta^{\VV}$. 
	\end{theorem}
	
	\section{Normal, Compatible, and Partial Connections}\label{sec:normal-compatible}
	
	In the Riemannian setting,  the Levi-Civita connection is the unique torsion-free affine connection that is compatible with the Riemannian metric. In the sub-Riemannian setting there exists no analogously canonical construction of an affine connection in general. In particular, unless $\HH$ is integrable, there exists no torsion-free affine connection on $M$ whose parallel transport preserves $\HH$. Nonetheless, there exist (at least) three natural notions of affine connections in the sub-Riemannian setting that each resemble important properties of the Riemannian case: normal, compatible, and partial connections. While none of these connections depend on the choice of a complement $\VV$ per se, such a choice yields a uniqueness result for partial connections and second-order retractions for normal and compatible connections, respectively. All of these connections have been studied in the literature before; however, a comprehensive overview highlighting their (dis)similarities appears to be missing. We therefore begin by reviewing normal connections.
	
	Let $\nabla: \Gamma(TM) \times \Gamma(TM) \to \Gamma(TM)$ 
	be an arbitrary affine connection on $M$. As usual, one defines 
	$ (\nabla_{X}\alpha)(Y) := X(\alpha(Y)) - \alpha(\nabla_{X}Y)$ for all $X ,Y\in \Gamma(TM)$ and all $\alpha \in \Gamma(T^{*}M)$.
	Given a sub-Riemannian structure $(\HH, h)$, define the \emph{``co-connection''} by
	$$
	\nabla_{\alpha} \beta := \nabla_{\alpha^{\sharp}} \beta \quad \text{for $\alpha, \beta \in \Gamma(T^{*}M)$} \ ,
	$$
	and denote its symmetric part by
	$$
	S^{\nabla}(\alpha, \beta) := \frac12 \left(\nabla_{\alpha}\beta + \nabla_{\beta}\alpha\right)  \ .  
	$$
	Call a curve $\alpha:[a,b]\to T^{*}M$ \emph{autoparallel}  if $\nabla_{\alpha(t)} \alpha(t) = 0$ for all $t$. It is straightforward to verify that two affine connections $\nabla$ and $\tilde\nabla$ on $M$ give rise to the same autoparallel curves if and only if their co-connections have the same symmetric part, i.e., if and only if $S^{\nabla} = S^{\tilde\nabla}$. We are interested in those connections $\nabla$ for which autoparallel curves are exactly the integral curves of the Hamiltonian vector field $X_{H}$ considered on Section~\ref{sec:normal-geodesics}, i.e., for which
	\begin{align}\label{eq:def-norm-conn}
	\nabla_{\alpha(t)} \alpha(t) = 0 \quad \iff \quad \dot \alpha(t) = X_{H}|_{\alpha(t)} \ . 
	\end{align}
	Following~\cite{LANGEROCK2003203}, we make the following definition:
	\begin{definition}[Normal connection]\label{def:def-norm-conn}
		An affine connection on $M$ is called $\HH$-\emph{normal} if autoparallel curves are the integral curves of the Hamiltonian vector field $X_{H}$, i.e., if~\eqref{eq:def-norm-conn} is satisfied.
	\end{definition}
\begin{remark}\label{rem:existence-normal-conn}
For a given sub-Riemannian structure, normal connections always exist, see, e.g., Proposition 18 in~\cite{LANGEROCK2003203}, but they are not unique in general.
\end{remark}

	One can verify that a connection $\nabla$ is normal if and only if the symmetric part satisfies
	$$
	S^{\nabla}(\alpha, \beta) = \frac 12 \left(\mathcal{L}_{\alpha^{\sharp}}\beta + \mathcal{L}_{\beta^{\sharp}}\alpha - d (g(\alpha, \beta)) \right)\ ,
	$$
	where $\mathcal{L}$ denotes the (usual) Lie derivative on smooth manifolds. For a proof, see, e.g., Theorem 17 in~\cite{LANGEROCK2003203}. 
	In local coordinates the symmetric part $S^{\nabla}(\alpha, \beta)$ of a normal connection is uniquely determined by the (dual) Christoffel symbols  $\Gamma^{ij}_{k}$ defined in~\eqref{eq:normal-geod-Hamiltonian-2}. 
	
\begin{remark}\label{rem:normal-peculiar}
As a word of caution we remark that the nomenclature ``normal connection'' might be somewhat misleading, since  geodesics arising from normal connections (i.e., curves that satisfy $ \nabla_{\dot \gamma(t)}\dot\gamma(t) = 0$) are in general \emph{not} normal geodesics. Indeed, let $\alpha(t)$ be autoparallel with respect to a normal connection $\nabla$, and let $\gamma(t)$ be the corresponding normal geodesics, i.e., $\dot\gamma(t)=\alpha^{\sharp}(t).$ Then
\begin{align*}
\nabla_{\dot \gamma(t)}\alpha(t) = 0 \ \text{(by construction)} \quad \text{but} \quad   \nabla_{\dot \gamma(t)}\dot\gamma(t) \neq 0 \ \text{(in general)}\ .
\end{align*}
Below we will show, however, that under certain conditions geodesics arising from normal connections (i.e., curves that satisfy $ \nabla_{\dot \gamma(t)}\dot\gamma(t) = 0$) yield second-order horizontal retractions. 
\end{remark}
Another class of connections that one naturally considers for sub-Riemannian structures are \emph{compatible} connections:
	
	\begin{definition}[Compatible connection]\label{def:def-comp-conn}
		An affine connection $\nabla$ on $M$ is called $\HH$-\emph{compatible} if (i) parallel transport with respect to $\nabla$ preserves $\HH$, which is equivalent to requiring that 
		$$
		\nabla_{X}Y \in \Gamma(\HH)\quad \forall X\in \Gamma(TM), \,  \forall Y \in \Gamma(\HH) \ , \quad \text{and}
		$$ 
		(ii) parallel transport preserves the metric $h$ on $\HH$, which is equivalent to requiring that 
		$$
		X(h(Y,Z)) = h(\nabla_{X}Y, Z) + h(Y, \nabla_{X}Z) \quad \forall X\in \Gamma(TM), \,  \forall Y,Z\in \Gamma(\HH) \ .
		$$
	\end{definition}
	
	\begin{remark}\label{rem:comp-conn-prop}
		A connection $\nabla$ is compatible with $(\HH, h)$ if and only if  $\nabla g \equiv 0$, i.e., the cometric $g$ is parallel, which in turn is equivalent to requiring that $(\nabla_{\alpha} \beta)^{\sharp}=  \nabla_{\alpha^{\sharp}} \beta^{\sharp}$ for all $\alpha, \beta \in \Gamma(T^{*}M)$, see, e.g.,~\cite{CHENG2021112387,Molina-Grong-2017,Grong2020AffineCA,Grong-Thalmaier2019}. 
	\end{remark}
Let the torsion of an affine connection be (as usual) defined by 
$$T(X,Y) =  \nabla_{X}Y  -  \nabla_{Y}X - [X,Y] \ .$$	
In order to relate normal connections to compatible connections, we require the following definition:
	
	\begin{definition}[Adjoint connection]\label{def:adjoint-connection}
		For an arbitrary affine connection $\nabla$ on $M$, define its \emph{adjoint} connection by 
		$$
		\hat \nabla_{X}Y:= \nabla_{X}Y - T(X,Y) \ ,
		$$
		where $T$ is the torsion of $\nabla$. 
	\end{definition}
	The torsion $\hat T$ of the adjoint connection $\hat \nabla$ satisfies $\hat T(X,Y) = - T(X,Y)$; therefore, the double adjoint satisfies $\hat{\hat{\nabla}}=\nabla$. 

\begin{remark}\label{rem:same-geodesics}
Notice that an affine connection $\nabla$ has the same geodesics as its adjoint connection $\hat \nabla$ since $\nabla_{\dot \gamma}\dot \gamma = 0$ if and only if $\hat \nabla_{\dot \gamma}\dot \gamma = 0$.
\end{remark}
	
	With these notions one obtains the following equivalence result between normal and compatible connections. 
	
	\begin{proposition}\label{prop:equiv-normal-compatible}
		Let $(\HH, h)$ be a sub-Riemannian structure on $M$. Then an affine connection on $M$ is $\HH$-normal if and only if its adjoint connection is $\HH$-compatible, and vice-versa.
	\end{proposition}

\begin{proof}
		Although this result is known (see, e.g., Proposition 3.1 in~\cite{Grong2020AffineCA}), we provide a proof for convenience. Let $\nabla$ be an $\HH$-normal connection with adjoint connection $\hat \nabla$. We need to show that $\hat \nabla$ is $\HH$-compatible. 
		Theorem 17 in~\cite{LANGEROCK2003203} implies that normality of $\nabla$ is equivalent to 
		\begin{align*}
		\alpha(\nabla_{\beta^{\sharp}}X) + \beta(\nabla_{\alpha^{\sharp}}X) = \alpha([\beta^{\sharp},X])+ \beta([\alpha^{\sharp},X]) + X(g(\alpha, \beta)) \quad\forall X\in \Gamma(TM), \, \forall \alpha, \beta \in \Gamma(T^{*}M) \ .
		\end{align*}
		This, in turn, is equivalent to
		\begin{align}\label{eq:normal-conn-equiv}
		\alpha(\hat\nabla_{X}\beta^{\sharp}) + \beta(\hat\nabla_{X}\alpha^{\sharp}) = X(g(\alpha, \beta)) \quad\forall X\in \Gamma(TM), \, \forall \alpha, \beta \in \Gamma(T^{*}M) \ .
		\end{align}
		Now let $\beta \in \HH^{0}$, where the \emph{annihilator} $\HH^{0}$ of $\HH$ is defined by 
		$$
		\HH^{0} := \{ \eta \in T^{*}M \, : \, \text{$\eta(Y) = 0$ for all $Y \in \HH$}\} \ .
		$$
		Then $\beta^{\sharp}=0 = g(\alpha, \beta)$,  and hence~\eqref{eq:normal-conn-equiv} gives $\beta(\hat\nabla_{X}\alpha^{\sharp})=0$. Since $\beta \in \HH^{0}$ was arbitrary, one obtains that $\hat\nabla_{X}\alpha^{\sharp}\in \HH$, and since $\alpha$ was arbitrary, one additionally obtains that $\hat\nabla_{X}Y \in \HH$ for all $Y\in \HH$. Hence, parallel transport with respect to $\hat \nabla$ preserves $\HH$, and~\eqref{eq:normal-conn-equiv}  implies that $\hat \nabla$ is indeed $\HH$-compatible. 
		The converse direction can be proven similarly. 
	\end{proof}

\begin{remark}
It follows from Proposition~\ref{prop:equiv-normal-compatible} together with Remark~\ref{rem:existence-normal-conn} that compatible connections always exist. It follows furthermore that if there exists an affine connection that is both compatible \emph{and} normal, then the torsion of such a connection vanishes, which in turn implies that  $[X,Y]$ is horizontal for all $X, Y \in \Gamma(\HH)$. Therefore, such a connection can only exist if $\HH$ is integrable.
\end{remark}

		Neither normal nor compatible connections are unique in general. In order to obtain (at least a partial notion of) uniqueness, it has been observed in~\cite{CHENG2021112387} that it is useful to restrict to \emph{partial connections}. 
	
	\begin{definition}[Partial connection]\label{def:part-comp-conn}
		Let $(\HH, h)$ be a sub-Riemannian structure on $M$. A \emph{partial connections} is a bilinear mapping 
		of the form $\nabla: \Gamma(\HH) \times \Gamma(\HH) \to \Gamma(\HH)$ that additionally satisfies
		$$
		\nabla_{fX}Y = f \nabla_{X}Y \quad \text{and} \quad \nabla_{X}fY = f \nabla_{X}Y + (Xf) Y  \quad \forall \, f \in C^{\infty}(M) \ .
		$$
		A partial connection is called \emph{partially compatible} if one additionally has that
		$$
		X(h(Y,Z)) = h(\nabla_{X}Y, Z) + h(Y, \nabla_{X}Z) \quad   \forall \,X,Y,Z\in \Gamma(\HH) \ .
		$$
	\end{definition}
	 
One has the following existence and uniqueness result,  see~\cite{CHENG2021112387}:
	
	\begin{proposition}\label{prop:part-comp-conn}
		Let $(\HH, h)$ be a sub-Riemannian structure on $M$ with complement $\VV$. Then there exists a unique partial connection $\nabla^{\VV}$ that is partially compatible and whose torsion satisfies $T(\HH, \HH) \subset \VV$. 
	\end{proposition}

A minor adaptation of the proof of Proposition~\ref{prop:part-comp-conn} found in~\cite{CHENG2021112387} shows that the unique partially compatible connection $\nabla^{\VV}$ can be extended to a compatible connection:
  
\begin{proposition}\label{prop:existence-comp-extension}
Let $(\HH,h)$ be a sub-Riemannian structure with complement $\VV$. Then there exists an $\HH$-compatible connection $\nabla$ whose torsion satisfies $T(\HH, \HH) \subset \VV$.
\end{proposition}
\begin{proof}
Let $\nabla'$ be an arbitrary reference compatible connection. Denote its torsion by $T'$, and consider the decomposition $T' = T'_{\HH} + T'_{\VV} = \mathrm{pr}_{\HH}T + \mathrm{pr}_{\VV}T$. Define a linear operator 
$
\kappa: \Gamma(TM) \to \Gamma(\mathrm{End}(\HH))
$
by 
$$
h(\kappa(X) Y_{1}, Y_{2}) := \frac 12 \big( 
-h(T'_{\HH} (X,Y_{1}),Y_{2})
+ h(T'_{\HH} (Y_{1},Y_{2}),\mathrm{pr}_{\HH}X)
+h(T'_{\HH} (X,Y_{2}),Y_{1}) 
\big) \ ,
$$
where $X\in \Gamma(TM)$ is arbitrary and $Y_{1}, Y_{2}\in \Gamma(\HH)$.
Then clearly $\kappa$ is skew-symmetric, i.e., 
$
h(\kappa(X) Y_{1}, Y_{2}) + h(\kappa(X) Y_{2}, Y_{1}) = 0.
$
Define a linear operator 
$\tilde \kappa: \Gamma(TM) \to \Gamma(\mathrm{End}(TM))$ by extending $\kappa$ according to
$\tilde \kappa (X) Y = \kappa(X)Y$ if $Y \in \Gamma(\HH)$ and $\tilde \kappa (X) Y = 0$ if $Y \in \Gamma(\VV)$.
Then $$\nabla_{X}Y := \nabla'_{X}Y + \tilde\kappa({X})Y $$
defines an affine connection (since $\tilde\kappa({X})$ is a linear operator), which is indeed $\HH$-compatible (since $\nabla'$ is compatible and $\kappa({X})$ is skew-symmetric). 

It remains to show that the torsion of $\nabla$ satisfies $T(\HH, \HH) \subset \VV$. This follows from noticing that 
\begin{align*}\tag{i}
T(X,Y) = T'(X,Y) + \kappa(X)Y- \kappa(Y)X \quad \forall X,Y\in \Gamma(\HH)
\end{align*}
and (using the definition of $\kappa$) that
\begin{align*}\tag{ii}
h(\kappa(X)Y_{1}, Y_{2})- h(\kappa(Y_{1})X, Y_{2}) = -h(T'_{\HH} (X,Y_{1}),Y_{2}) \quad \forall X,Y_{1}, Y_{2}\in \Gamma(\HH) \ .
\end{align*}
Then (i) and (ii) yield that
\begin{align*}\tag{iii}
T(X,Y) = T'(X,Y) - T'_{\HH}(X,Y) = T'_{\VV}(X,Y) \in \VV \quad \forall X,Y\in \Gamma(\HH) \ ,
\end{align*}
which completes the proof.
\end{proof}

\begin{remark} \label{rem:compatible-or-normal}
The statement of Proposition~\ref{prop:existence-comp-extension} remains true if ``compatible connection'' is replaced by ``normal connection'', since the torsion of a normal connection is the negative of the torsion of its adjoint compatible connection. 
\end{remark}

Under certain circumstances, the result of Proposition~\ref{prop:existence-comp-extension} can be strengthened. To this end, consider the following definition:

	\begin{definition}[Metric-preserving complement]\label{def:metric-comp}
		Let $(\HH,h)$ be a sub-Riemannian structure with complement $\VV$. Then $\VV$ is called \emph{metric-preserving} if $\mathcal{L}^{\VV}_{Y} h= 0$ for all vertical fields $Y$, where $\mathcal{L}^{\VV}$ is defined as in Definition~\ref{def:hor-Lie}.  
	\end{definition}
	
	In other words, $\VV$ is metric-preserving if the flow along vertical vector fields preserves the horizontal metrics on $\HH$. With this notion at hand, Proposition~\ref{prop:existence-comp-extension} can be strengthened in the following sense; see~\cite{CHENG2021112387} for a proof:
	
\begin{proposition}\label{prop:existence-comp-extension-sharpened}
Let $(\HH,h)$ be a sub-Riemannian structure with metric-preserving complement $\VV$. Then there exists a compatible connection such that both $\HH$ and $\VV$ are preserved under parallel transport, and such that the torsion $T$ of $\nabla$ satisfies $T(\HH, \HH)\subset \VV$ and  $T(\HH, \VV)= 0$. 
\end{proposition}

Returning to the main theme of our exposition, we now discuss how the choice of a complement $\VV$ allows for expressing the sub-Laplacian $\Delta^{\VV}$ using partial connections (see Section~\ref{sec:horizontal-Laplacians}) and how certain normal and compatible connections yield second-order retractions for normal geodesics (see Section~\ref{sec:geodesics-revisited}).
	
\subsection{Sub-Laplacians revisited}\label{sec:horizontal-Laplacians}
	The horizontal divergence and sub-Laplacian introduced in Definition~\ref{def:hor-div-and-sub-Lap-hor} can conveniently be expressed using partially compatible connections:
	\begin{lemma}\label{lem:hor-Lie-hor-div}
		Let $(\HH,h)$ be a sub-Riemannian structure with complement $\VV$, and let $\nabla^{\VV}$ denote the unique partially compatible connection from Proposition~\ref{prop:part-comp-conn}. Then for all $X\in \Gamma(\HH)$,
		\begin{align*}
		\mathrm{div}^{\VV}(X) = \mathrm{tr}_{\HH}(\nabla^{\VV}X) := \sum_{i=1}^{k} h(\nabla^{\VV}_{E_{i}}X, E_{i}) \ ,
		\end{align*}
		where $(E_{1}, \dots, E_{k})$ is a (local) orthonormal horizontal frame of $\HH$ and 
		$\mathrm{div}^{\VV}$ denotes the horizontal divergence defined by~\eqref{eq:hor-div}. In particular, one has that
		\begin{align*}
		\Delta^{\VV}f = \mathrm{tr}_{\HH}(\nabla^{\VV}(df)^{\sharp})  \ ,
		\end{align*}
		where $\Delta^{\VV}$ is the sub-Laplacian defined by~\eqref{eq:hor-Lap}.
	\end{lemma}
	\begin{proof}
		Fix $X\in \Gamma(\HH)$, and consider a  (local) orthonormal horizontal frame $(E_{1}, \dots, E_{k})$ in $\HH$ that (without loss of generality) locally satisfies $\nabla^{\VV}_{X}E_{i}=0$ and thus 
		$$
		\mathrm{pr}_{\HH} [X,E_{i}] = \nabla^{\VV}_{X}E_{i} - \nabla^{\VV}_{E_{i}}X = - \nabla^{\VV}_{E_{i}}X \ ,
		$$
		where for the first equality we have used that the torsion of $\nabla^{\VV}$ lies in $\VV$ by construction. Using the definition of $\omega_{\HH}$ from~\eqref{eq:hor-volume-form}, one obtains
		\begin{align*}
		\left(\mathcal{L}_{X}^{\VV}\omega_{\HH}\right)(E_{1}, \dots, E_{k})
		&= - \sum_{i=1}^{k} \omega_{\HH}(E_{1}, \dots, \mathrm{pr}_{\HH}[X, E_{i}], \dots, E_{k}) = \sum_{i=1}^{k} \omega_{\HH}(E_{1}, \dots, \nabla^{\VV}_{E_{i}}X, \dots, E_{k})\\
		&=\left(\sum_{i=1}^{k} h(\nabla^{\VV}_{E_{i}}X, E_{i})\right) \omega_{\HH}(E_{1}, \dots, E_{k}) \ ,
		\end{align*}
		which proves the claim.
	\end{proof}
	
\subsection{Geodesic random walks revisited}\label{sec:geodesics-revisited}
	Let $\nabla$ be an $\HH$-compatible or an $\HH$-normal connection. We are interested in the question of when a geodesic random walk with respect to $\nabla$ (i.e., walking along curves that satisfy $\nabla_{\dot \gamma}\dot \gamma = 0$) converges to the correct horizontal Brownian motion. 
	
With regards to the results of Section~\ref{sec:rb-rw}, we answer this question by relating geodesics arising from compatible (or normal) connections to second-order retractions. 
In order to do so, we require the notion of the \emph{exponential map} of an affine connection $\nabla$, by which we mean the map $\Exp_{x}:T_{x}M\to M$, where $\Exp_{x}(u):=\gamma(1)$ and $\gamma$ is the $\nabla$-geodesic  (i.e., the curve satisfying $\nabla_{\dot \gamma}\dot \gamma =0$) with initial conditions $(x,u)$. With this notation at hand, we record the following result:
\begin{proposition}\label{prop:comp-norm-second-order}
Let $(\HH,h)$ be a sub-Riemannian structure with complement $\VV$, and let $\nabla$ be an $\HH$-compatible (or $\HH$-normal) connection.  Then the exponential map $\Exp\lvert_{\HH}$ of $\nabla$, when restricted to $\HH$, is a second-order horizontal retraction if and only if the torsion of $\nabla$ satisfies  $T(\HH, \HH) \subset \VV$.
\end{proposition}

\begin{proof}
According to Remark~\ref{rem:same-geodesics}, it suffices to consider compatible connections. We only show the backward direction of the if-and-only-if statement; the proof of the forward direction is similar. 

Let $\nabla$ denote an $\HH$-compatible affine connection on $M$, whose torsion (by assumption) satisfies $T(\HH, \HH) \subset \VV$, and let $\hat \nabla = \nabla - T(\cdot, \cdot)$ be the adjoint $\HH$-normal connection; see Proposition~\ref{prop:equiv-normal-compatible}. Furthermore, let  $\gamma:[0,\varepsilon) \to M$ be the normal geodesic with horizontal initial conditions $(x,u)$. Denote by $\alpha(t)$ the Hamiltonian curve $\alpha :[0,\varepsilon)\to T^{*}M$ corresponding to the normal geodesic $\gamma$, i.e., $\dot\alpha(t) = X_{H}|_{\alpha(t)}$, $\alpha(t)^{\sharp}= \dot\gamma (t)$, and $\alpha(0)|_{\VV} = 0$.
		Since $\hat \nabla$ is normal, it holds that
		\begin{align*}
		\hat\nabla_{\dot \gamma} \alpha \equiv 0 \ .
		\end{align*}
		Now let $X\in \Gamma(\HH)$ be a horizontal vector field. 
		Then, using that $\nabla$ is compatible, we obtain that
		\begin{align*}
		0=\left( \hat\nabla_{\dot \gamma} \alpha \right)(X) 
		&= \dot \gamma \left( \alpha (X)\right) - \alpha \left( \hat \nabla_{\dot \gamma}X\right) 
		=  \dot \gamma \left( h(\alpha^{\sharp}, X)\right) - \alpha \left( \hat \nabla_{\dot \gamma}X\right) \\
		&= h(\nabla_{\dot \gamma}\alpha^{\sharp}, X)  + h(\alpha^{\sharp}, \nabla_{\dot \gamma}X) - \alpha \left( \hat \nabla_{\dot \gamma}X\right) \\
		&= h(\nabla_{\dot \gamma}\alpha^{\sharp}, X)  + \alpha \left( \nabla_{\dot \gamma}X\right) - \alpha \left( \hat \nabla_{\dot \gamma}X\right)\\
		&= h(\nabla_{\dot \gamma}\dot\gamma, X)  + \alpha \left( T(\dot \gamma, X)\right) \ .
		\end{align*}
		Evaluating the right hand side at $t=0$, and using that $ T(\dot \gamma, X)$ is vertical as well as $\alpha(0)|_{\VV}=0$, we obtain that 
		$$
		h(\nabla_{\dot \gamma}\dot\gamma, X)|_{\gamma(0)} = 0 \ .
		$$
		Using that $X$ was an arbitrary horizontal field shows that $\nabla_{\dot \gamma}\dot \gamma|_{\gamma(0)} = 0$. Using Remark~\ref{rem:second-order-agreement-connection} then proves the claim.
\end{proof}	

\begin{corollary}\label{cor:second-order-retractions-partial-conn}
Let $(\HH,h)$ be a sub-Riemannian structure with complement $\VV$, and let $\nabla^{\VV}$ be the unique partially compatible connection whose torsion satisfies $T(\HH,\HH)\subset \VV$. Then $\nabla^{\VV}$-geodesics with horizontal initial conditions provide second-order horizontal retractions. 
\end{corollary}
\begin{proof}
Existence of $\nabla^{\VV}$ follows by Proposition~\ref{prop:part-comp-conn}. By Proposition~\ref{prop:existence-comp-extension}, $\nabla^{\VV}$ can be extended to a compatible connection whose torsion satisfies $T(\HH,\HH)\subset \VV$. Then Proposition~\ref{prop:comp-norm-second-order} implies the claim.
\end{proof}
Corollary~\ref{cor:second-order-retractions-partial-conn} and Theorem~\ref{thm:convergence-processes-ret2} immediately yield the following result:

\begin{theorem}\label{thm:scaling-limit-nabla-geod-walk}
Let $M$ be compact, and let $(\HH,h)$ be a \emph{bracket-generating} sub-Riemannian structure  on $M$ with complement $\VV$. Let $\nabla^{\VV}$ be the unique partially compatible connection whose torsion satisfies $T(\HH, \HH)\subset \VV$. Then the horizontal geodesic random walk along $\nabla^{\VV}$-geodesics (with horizontal sampling steps) converges in distribution to the horizontal Brownian motion generated by $\Delta^{\VV}$ as $\varepsilon \to 0$. 
\end{theorem}

\begin{remark}
Notice that the claim of Theorem~\ref{thm:scaling-limit-nabla-geod-walk} remains valid if instead of the unique partially compatible connection with $T(\HH, \HH)\subset \VV$ one considers compatible or normal connections with $T(\HH, \HH)\subset \VV$.
\end{remark}

For completeness we point out that the result of Proposition~\ref{prop:comp-norm-second-order} can be sharpened for metric-preserving complements: 	
	\begin{proposition}\label{prop:comp-norm-full-order}
		Let $(\HH,h)$ be a sub-Riemannian structure with metric-preserving complement $\VV$. Consider a compatible connection $\nabla$ from Proposition~\ref{prop:existence-comp-extension-sharpened}. Then any normal geodesic $\gamma:[0,\varepsilon) \to M$ with horizontal initial conditions satisfies
		$
		\nabla_{\dot \gamma}\dot \gamma \equiv  0 \ .
		$
In other words, for such connections $\nabla$, normal geodesics agree with $\nabla$-geodesics. 	\end{proposition}

	\begin{proof}
		Let $\nabla$ be a compatible connection provided by Proposition~\ref{prop:existence-comp-extension-sharpened}. Then $\nabla$  preserves $\HH$ and $\VV$, and its torsion satisfies $T(\HH, \HH)\subset \VV$ and $T(\HH, \VV) = 0$.  Let $\bar \gamma$ be the unique geodesic that satisfies $\nabla_{\dot{\bar\gamma}}\dot{\bar \gamma} \equiv  0$ with initial conditions $x \in M$ and $u\in \HH_{x}$. Denote by $\alpha(t) \in T^{*}_{\gamma(t)}M$ the unique $1$-form along $\bar \gamma$ with $\alpha^{\sharp}(t) = \dot {\bar\gamma}(t)$ and $\alpha(t)|_{\VV}\equiv 0$.
		Furthermore, let $\hat \nabla = \nabla - T(\cdot, \cdot)$ be the adjoint normal connection corresponding to $\nabla$.  We claim that 
		\begin{align}\label{eq:comp-norm-full-order-suffice}
		\hat \nabla_{\dot{ \bar\gamma}} \alpha \equiv 0 \ .
		\end{align}
		Proving~\eqref{eq:comp-norm-full-order-suffice} suffices for proving the proposition, since then $\alpha$ is the unique autoparallel curve with the requisite horizontal initial conditions.  In order to show~\eqref{eq:comp-norm-full-order-suffice}, we consider the application of $\hat \nabla _{\dot \gamma}\alpha$ to horizontal and vertical vector fields separately.
		
		First let $X\in \Gamma(\HH)$ be a horizontal vector field. Then,
		\begin{align*}
		\left( \hat\nabla_{\dot{ \bar\gamma}}  \alpha \right)(X) 
		&= {\dot{ \bar\gamma}} \left( \alpha (X)\right) - \alpha \left( \hat \nabla_{\dot{ \bar\gamma}} X\right) 
		=  {\dot{ \bar\gamma}}  \left( h(\alpha^{\sharp}, X)\right) - \alpha \left( \hat \nabla_{\dot{ \bar\gamma}} X\right) \\
		&= h(\nabla_{\dot{ \bar\gamma}} \alpha^{\sharp}, X)  + h(\alpha^{\sharp},  \nabla_{\dot{ \bar\gamma}} X) - \alpha \left( \hat \nabla_{\dot{ \bar\gamma}} X\right) \\
		&= h(\nabla_{\dot{ \bar\gamma}} {\dot{ \bar\gamma}} , X)  + \alpha \left( T({\dot{ \bar\gamma}} , X)\right)  = \alpha \left( T({\dot{ \bar\gamma}} , X)\right)= 0 \ ,
		\end{align*}
		where for the last equality we have used that $T({\dot{ \bar\gamma}} , X)$ is vertical. 
		
		Next let $Y\in \Gamma(\VV)$ be a vertical vector field. Using that $\alpha|_{\VV} \equiv 0$ and  $\nabla_{\dot{\bar\gamma}} Y \in \VV$, 
		we obtain
		\begin{align*}
		\left( \hat\nabla_{\dot{\bar \gamma}} \alpha \right)(Y) 
		= - \alpha \left( \hat \nabla_{\dot{\bar \gamma}}Y\right) 
		=  \alpha \left( T ({\dot{\bar \gamma}}, Y) - \nabla_{\dot{\bar \gamma}}Y\right) 
		= \alpha \left( T ({\dot{\bar \gamma}}, Y) \right)   = 0 \ ,
		\end{align*}
		where for the last equality we used that $T(\HH, \VV) = 0$. This proves ~\eqref{eq:comp-norm-full-order-suffice}.
	\end{proof}

	\section{Frame bundles}\label{sec:frame-bundles} 	
We apply the previous results to the specific case of frame bundles. Our motivation for considering frame bundles is twofold. For one, frame bundles constitute a prime example of sub-Riemannian structures that additionally come with a canonical choice of a complement $\VV$. Secondly, frame bundles over manifolds with affine connections provide a natural setting for treating anisotropic (or biased) Brownian motion; see Theorem~\ref{thm:conv-ret-frame-bundle} and Remark~\ref{rem:frame-bundle-anisotropic-BM}.

Given a smooth manifold $M$ of dimension $n$, the \emph{frame bundle} $\pi: \FF\to M$ is the $\mathrm{GL}(n;\mathbb{R})$ principal bundle over $M$ whose fibers $\mathcal{F}_{x}$ consist of all ordered bases of $T_{x}M$. One frequently identifies elements $F\in \FF_{x}$ with linear isomorphisms 
	$$
	F: \mathbb{R}^{n}\to T_{x}M \ ,
	$$
	which map the standard basis of $\mathbb{R}^{n}$ to the basis of $T_{x}M$ corresponding to a given frame $F$. Then the linear group $\mathrm{GL}(n;\mathbb{R})$ naturally acts on the right on elements of $\FF$.  As usual, let the canonical \emph{vertical bundle} be defined by $\VV := \ker (d\pi) \subset T\FF$. 
	
	Given an \emph{arbitrary} affine connection $\nabla$ on $M$, there exists a natural sub-Riemannian structure $(\HH^{\nabla}, h^{\nabla})$ on $\FF$ with complement $\VV$, constructed as follows. As usual, let $\HH^{\nabla} \subset T\FF$ be the horizontal 
	subbundle, defined by requiring that horizontal curves $\gamma: [a,b]\to \FF$ correspond to parallel transport along the projected curves $\pi\circ \gamma$ with respect to the given affine connection $\nabla$ on $M$. Then, in particular, the restriction $$d\pi|_{\HH^{\nabla}_{F}}: \HH^{\nabla}_{F} \to T_{\pi(F)}M$$  is a linear isomorphism of vector spaces for all $F\in \FF$. Clearly, one has $\HH^{\nabla} \oplus \VV = T\FF$, and $\VV$ is indeed a complement of $\HH^{\nabla}$.
	
	In order to construct a sub-Riemannian metric $h^{\nabla}$ on $\HH^{\nabla}$, consider the \emph{solder form}, sometimes also called the tautological $1$-form, $\theta \in \Omega^{1}(\FF; \mathbb{R}^{n})$ whose values on a tangent vector $u\in T_{F}\FF$ are the coefficients of the projected vector $d\pi(u) \in T_{\pi(F)}M$ in the frame given by $F$. Put differently, if we regard $F$ as a linear isomorphism $\mathbb{R}^{n}\to T_{\pi(F)}M$, then $\theta_{F}(u) = F^{-1}(d\pi(u))$. In particular, $\theta(\VV) \equiv 0$, and the restriction
\begin{align}\label{eq:solder-form-restricted}
	\phi_{F}:=  \theta_{F}|_{\HH^{\nabla}_{F}}:\HH^{\nabla}_{F} \to \mathbb{R}^{n} 
\end{align}
	yields an isomorphism of vector spaces. This allows for defining a canonical metric on $\HH^{\nabla}$:
	
	\begin{definition}[Sub-Riemannian structure on frame bundle]\label{def:sub-Riemannian-frame-bundle}
Define a sub-Riemannian metric $h^{\nabla}$ on $\HH^{\nabla}$ by requiring that $\phi_{F}$ is an \emph{isometry}, where $\mathbb{R}^{n}$ is equipped with the standard Euclidean metric. 
	\end{definition}
	In other words,  $h^{\nabla}$ is the sub-Riemannian metric on $\HH^{\nabla}$ constructed as follows: Regarding each frame $F\in \FF$ as a basis of $T_{\pi(F)}M$, one defines the \emph{horizontal lift} of the attendant  basis vectors to $\HH^{\nabla}_{F}$ to be an \emph{orthonormal frame} in $\HH^{\nabla}_{F}$.
	
	Consider also the  \emph{connection $1$-form} $\omega \in \Omega^{1}(\FF; \mathfrak{gl}(n))$ corresponding to the horizontal subbundle $\HH^{\nabla}\subset \FF$ and taking values in the Lie algebra $\mathfrak{gl}(n)$. Recall that $\omega$ is defined by requiring that $\omega(\HH^{\nabla})\equiv 0$ and
	$$
	\omega_{F}\left(\frac{d}{dt} (F\cdot g(t))\big\rvert_{t=0}\right) = \frac{d}{dt} g(t)\big\rvert_{t=0} \in \mathfrak{gl}(n)
	$$
	for any smooth curve $g: [0, \varepsilon) \to \mathrm{GL}(n;\mathbb{R})$ with $g(0) = \mathrm{Id}$. In particular, the restriction
	$$
	\psi_{F}:=  \omega_{F}|_{\VV_{F} }:\VV_{F} \to \mathfrak{gl}(n) 
	$$
	is an isomorphism of vector spaces. Together with the solder form $\theta$, this allows for defining a canonical compatible connection for the sub-Riemannian structure  $(\HH^{\nabla}, h^{\nabla})$:
	
	\begin{definition}[Compatible connection on frame bundle]\label{def:comp-conn-frame-bundle}
		Let $\nabla$ be an affine connection on $M$, let $\gamma:[0,\varepsilon)\to \FF$ be an arbitrary smooth curve in the frame bundle with $\gamma(0)=F$, and let $Y(t) \in T_{\gamma(t)}\FF$ be a vector field along $\gamma$. Define an affine connection $\bar\nabla$ on $\FF$ by 
		\begin{align}\label{eq:part-conn-framebundle}
		\left(\bar\nabla_{\dot\gamma} Y \right)\big\rvert_F:= 
		\left(\phi_{F}\right)^{-1}\left(\frac{d}{dt} \theta_{\gamma(t)}(Y(t)) \big\rvert_{t=0} \right)  +
		\left(\psi_{F}\right)^{-1}\left(\frac{d}{dt} \omega_{\gamma(t)}(Y(t))\big\rvert_{t=0} \right)  \ .
		\end{align}
	\end{definition}
	
	\begin{remark}
	   The connection $\bar \nabla$ corresponds to the Cartan connection $(\theta +\omega)$ regarded as an affine connection on the frame bundle $\FF$. We refer the reader to, e.g.,~\cite{Beschastnyi2021,Grong2022-notes,MORIMOTO2008} for more details on the relation between Cartan connections and sub-Riemannian structures.
	\end{remark}
	
In particular, starting with a \emph{horizontal} vector $u_{0} \in \HH^{\nabla}_{f}$, parallel transport $u(t)$ with respect to $\bar \nabla$ along $\gamma(t)$ corresponds to insisting that $u(t)$ stay horizontal and that $\phi_{\gamma(t)}(u(t)) = \phi_{f}(u_{0})$. Therefore, parallel transport with respect to $\bar \nabla$ preserves $\HH^{\nabla}$ together with the inner product $h^{\nabla}$. It follows that $\bar \nabla$ is indeed a \emph{compatible connection} in the sense of Definition~\ref{def:def-comp-conn}. 
	
Moreover, if $X$ and $Y$ are vector fields on $M$ with horizontal lifts $\bar X$ and $\bar Y$, respectively, then
\begin{align}\label{eq:lifts-commute-with-connection}
	\bar\nabla_{\bar X} \bar Y = \overline{\nabla_{X} Y} \ ,
\end{align}
where the right hand side denotes the horizontal lift of $\nabla_{X} Y$. This follows from the fact that a vector that is parallel transported along a curve in the base manifold $M$ does not change its coordinates in any frame that is parallel transported along the same curve.

\begin{remark}
The connection $\bar \nabla$ defined by~\eqref{eq:part-conn-framebundle} is a metric connection for the bundle metric considered by Marathe, see~\cite{Marathe1972}, i.e., parallel transport with respect to $\bar \nabla$ preserves this metric. Other extensions of $\nabla$ to affine connections on the frame bundle have been considered in the literature, most prominently the \emph{canonical} and \emph{horizontal} lifts, see, e.g., \cite{Catuogno2008,Cordero1989,Mok1978}. However, different from $\bar\nabla$, the canonical lift does not satisfy an equality equivalent to \eqref{eq:lifts-commute-with-connection}, and neither the canonical lift nor the horizontal lift preserves the sub-Riemannian metric $h^{\nabla}$, i.e., these lifts do not provide compatible connections for our choice of sub-Riemannian metric.
\end{remark}

Using \eqref{eq:lifts-commute-with-connection}, a straightforward calculation shows that if the original connection $\nabla$ on $M$ is torsion-free, then the torsion $\bar T$ of $\bar \nabla$ satisfies $\bar T(\HH^{\nabla}, \HH^{\nabla}) \subset \VV$. 
Therefore, by Proposition~\ref{prop:part-comp-conn}, the restriction of $\bar \nabla$ to $\HH^{\nabla}$ is the unique partially compatible connection with this property. By Corollary~\ref{cor:second-order-retractions-partial-conn} $\bar \nabla$-geodesics with horizontal initial conditions are second-order horizontal retractions. Moreover, 
by \eqref{eq:lifts-commute-with-connection}, it follows that $\bar\nabla$-geodesics on $\FF$ with horizontal initial conditions are in one-to-one correspondence to parallel transport of frames along $\nabla$-geodesics on the base manifold $M$. This implies the following result:
\begin{proposition}\label{prop:parallel-transport-frames-second-order-ret}
Let $\nabla$ be a torsion-free affine connection on $M$. Then parallel transport of frames along $\nabla$-geodesics on $M$ yields a second-order horizontal retraction on the frame bundle.
\end{proposition}
Notice that the claim of Proposition~\ref{prop:parallel-transport-frames-second-order-ret} remains valid for any smooth subbundle $\mathcal{E} \subseteq \FF$ (principal or not) as long as $\mathcal{E}$ is preserved under parallel transport along curves in $M$. Orthogonal and anisotropic frame bundles over Riemannian manifolds (see Section~\ref{sec:algo}) provide examples of such subbundles. Upon restricting to \emph{compact} subbundles, Proposition~\ref{prop:parallel-transport-frames-second-order-ret} and Theorems~\ref{thm:convergence-generators-ret2} and~\ref{thm:scaling-limit-nabla-geod-walk} yield the following result:

\begin{theorem}\label{thm:conv-ret-frame-bundle}
Let $M$ be an $n$-dimensional smooth, compact manifold without boundary, and let $\nabla$ be a torsion-free affine connection on $M$. Let $\mathcal{E} \subseteq \FF$ be a smooth, compact subbundle that is preserved under parallel transport along curves in $M$. Consider the following algorithm. Pick an initial position $x\in M$ and an initial frame $F: \mathbb{R}^{n}\to T_{x}M$ corresponding to an element in $\mathcal{E}_{x}$. Fix $\varepsilon>0$, and keep repeating the following two steps:

\begin{enumerate}
\item[\mylabel{fb-sample}{($\bar{\mathrm{S}}$)}] Sample a unit vector $\bar u \in \mathbb{S}^{n-1} \subset \mathbb{R}^{n}$ uniformly with respect to the Euclidean metric on $\mathbb{R}^{n}$, and let $u:= F(\bar u) \in T_{x}M$.

\item[\mylabel{fbwalk}{($\bar{\mathrm{W}}$)}] 
Follow the $\nabla$-geodesic with initial conditions $(x,u)$ for time $\varepsilon \sqrt{n}$, and parallel transport $F$ along this geodesic using $\nabla$. Update $x$ to be the endpoint, and update $F$ to be the final frame along this geodesic.
\end{enumerate}
Then the infinitesimal generator of the resulting walk on the bundle $\mathcal{E}$ converges to the sub-Laplacian $\Delta^{\VV}$ on $\mathcal{E}$ as $\varepsilon \to 0$. Moreover, if the sub-Riemannian structure $\HH^{\nabla}$ is bracket-generating for $\mathcal{E}$, 
then the stochastic process resulting from the above walk converges to the horizontal Brownian motion generated by $\Delta^{\VV}$. 
\end{theorem}
	
\begin{remark}\label{rem:frame-bundle-anisotropic-BM}
The Brownian motion on the frame bundle considered in Theorem~\ref{thm:conv-ret-frame-bundle} corresponds to \emph{rolling without slipping or twisting} (also known as \emph{anti-developing}) $M$ along a path given by standard Brownian motion in $\mathbb{R}^{n}$. Indeed, the corresponding (Stratonovich) stochastic differential equation on $\mathcal{E}$ takes the form
\begin{align}\label{eq:anisotropic-BM-frame-bundle}
d F_{t} = \phi_{F_{t}}^{-1}\circ d B_{t} \quad \text{with initial condition $F_{0}\in \mathcal{E}$} \ ,
\end{align}
where $B_{t}$ denotes standard Brownian motion on $\mathbb{R}^{n}$ and $\phi_{F}$ is defined in~\eqref{eq:solder-form-restricted}. Notice that this equation generalizes an\-iso\-tropic (or biased) Brownian motion in Euclidean space $\mathbb{R}^{n}$ defined by the stochastic differential equation 
\begin{align}\label{eq:anisotropic-BM-Euclidean}
dX_{t}= A \circ dB_{t} \quad \text{for some fixed $A  \in \mathrm{GL}(n)$} \ ,
\end{align}
with bias given by the frame $A  \in \mathrm{GL}(n)$.
\end{remark}

\begin{figure*}[t]
	\centering
	\includegraphics[width=0.65\textwidth]{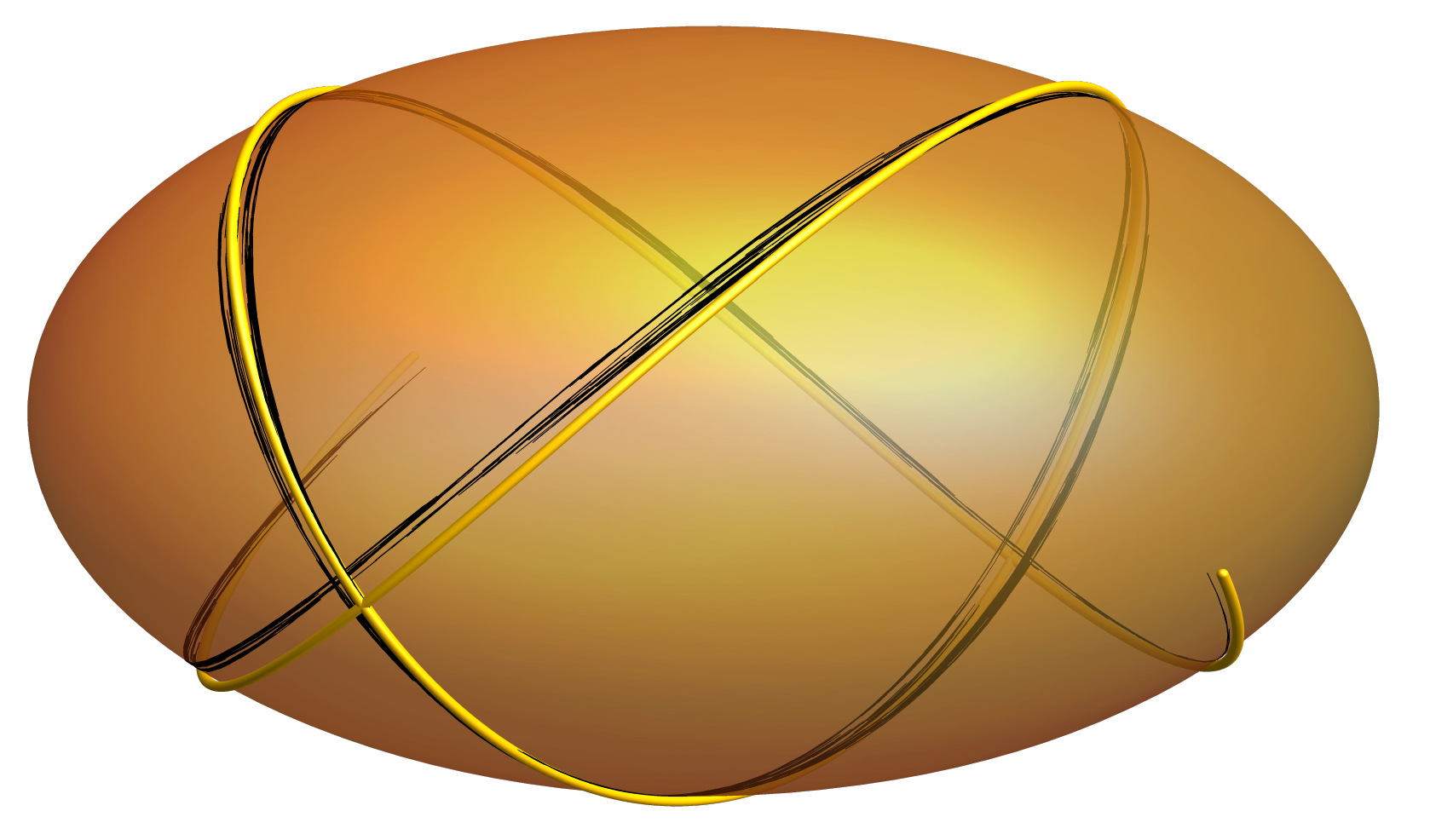}
	\caption{Anisotropic random walk resulting from Theorem~\ref{thm:conv-ret-frame-bundle}, projected back from an anisotropic frame bundle $\mathcal{E}:=\FF_{A}$ (see Definition~\ref{def:aniso-frame-bundle}) to an underlying Riemannian base manifold $M$ equipped with the Levi--Civita connection. Notice how this random walk (depicted in black with $20.000$ steps) stays close to the geodesic (depicted as a thin tube) with initial conditions corresponding to the highest anisotropy direction. For details, see ~\eqref{eq:ret-param-frame-A} from Section~\ref{sec:algo}).}
	\label{fig:anistropic-walk-ellipsoid}
\end{figure*}

The stochastic process on the frame bundle defined by~\eqref{eq:anisotropic-BM-frame-bundle} is a Markov process. However, the \emph{projected} process back to $M$ will no longer be a Markov process in general. Indeed, due to holonomy, the result of parallel transporting an initial frame $F_{0}$ around a closed path in $M$ will in general result in a frame different from $F_{0}$.
An equivalent way of viewing the failure of the projected process to be a Markov process on $M$ is to notice that 
the attendant sub-Laplacian $\Delta^{\VV}$ corresponding to the sub-Riemannian structure $(\HH^{\nabla}, h^{\nabla})$ on $\mathcal{E}$, i.e., the generator of the horizontal Brownian motion in $\mathcal{E}$, does not in general project to a meaningful operator on $M$. Figure~\ref{fig:anistropic-walk-ellipsoid} shows the result of a (highly) anisotropic random walk projected back to the base manifold $M$. Notice how the failure of this projected walk to be a Markov process on $M$ is exemplified by the two crossings at the front and back of the ellipsoid. 
For details of the attendant computation we refer to Section~\ref{sec:algo} and in particular to~\eqref{eq:ret-param-frame-A}.

\section{Computational aspects}\label{sec:algo}
Computing solutions to geodesic differential equations can be costly, which hampers an efficient algorithmic treatment of the geodesic random walks considered above. Therefore, in this section, we provide second-order retractions for the various geodesics considered in this article (i.e., normal geodesics or geodesics arising from compatible or normal connections). Then, the convergence results of Theorem~\ref{thm:convergence-generators-ret2} and Theorem~\ref{thm:convergence-processes-ret2} apply to
the resulting approximate geodesic walks. We will not explicitly repeat this fact for the three second-order retractions considered below. 
Throughout this section we work in single charts. As a notational convention, we decorate approximations of chart-based quantities with a tilde.

\subsection{A second-order retraction for normal geodesics}\label{ssec:numerics-general}
We start with a simple and explicit second-order approximation for normal geodesics:

\begin{proposition}\label{prop:2nd-order-ret-normal-geod}
	Using canonical coordinates in $T^{*}M$, let $(x(t),p(t))$ denote the normal geodesic with initial conditions $(x(0), p(0))$, i.e., the solution to the Hamiltonian ODE~\eqref{eq:normal-geod-Hamiltonian-1}. Then the curve
	\begin{align}\label{eq:simple-second-order-approx}
	\tag{$\Ret$-1}
	\tilde x^{i}(t) := x^{i}(0) + t g^{ij}(0)p_{j}(0) 
	+ \frac{1}{2} t^{2}\left(2g^{lk}(0)  \Gamma^{ij}_{k}(0) - g^{ik}(0) \Gamma^{lj}_{k}(0) \right)p_{l}(0)p_{j}(0)\ ,
	\end{align}
	with $g^{ij}$ and $\Gamma^{ij}_{k}$ defined in \eqref{eq:normal-geod-Hamiltonian-2}, provides a second-order approximation of $x(t)$ at $t=0$. 
\end{proposition}
\begin{proof}
	By definition, the normal geodesic $(x(t),p(t))$ is a solution to the Hamiltonian ODE~\eqref{eq:normal-geod-Hamiltonian-1}. Therefore, from $\dot x^{i}(t) = g^{ij}(t) p_{j}(t)$, one obtains that
	\begin{align*}
	\ddot x^{i}(t) 
	&= \dot g^{ij}(t) p_{j}(t) + g^{ij}(t) \dot p_{j}(t) \\
	&= 2\Gamma^{ij}_{k}(t) \dot x^{k}(t) p_{j}(t) + g^{ik}(t) \dot p_{k}(t) \\
	&= \left(2\Gamma^{ij}_{k}(t)g^{kl}(t) - g^{ik}(t) \Gamma^{lj}_{k}(t)\right) p_{l}(t)p_{j}(t) \ .
	\end{align*}
	The claim then follows from a direct application of Taylor's formula and by using that $g^{kl}= g^{lk}$.
\end{proof}

In order to obtain a second-order retraction with \emph{horizontal} initial conditions $(x,u)$, $u \in \HH_{x}$ from~\eqref{eq:simple-second-order-approx}, one needs to compute the requisite initial condition $p(0)$ corresponding to $u$. In other words, $p(0)$ needs to correspond to the $1$-form $\alpha_{0}\in T^{*}_{x}M$ with $\alpha_{0}^{\sharp}=u \in \HH_{x}$ and $\alpha_{0}|_{\VV}=0$ for a given complement $\VV$. This requires a linear solve in each iteration of the walk considered in Definition~\ref{def:ret-rand-walk}. This linear solve can be avoided by working with affine (e.g., compatible or normal) connections, which is what we consider next. 

\subsection{A second-order retraction for affine geodesics}\label{ssec:numerics-compatible-connection}
We now turn to second-order retractions for geodesics arising from  arbitrary affine connections. Compatible and normal connections on sub-Riemannian manifolds constitute a special case.

\begin{proposition}\label{prop:retraction-from-comp-geodesic}
	Let $M$ be a smooth manifold with affine connection $\nabla$, and let the attendant Christoffel symbols $\Gamma^{i}_{jk}$ be defined  by $\nabla_{\partial_{j}}\partial_{k}= \Gamma^{i}_{jk}\partial_{i}$. Then, given initial conditions $(x,u)$, with $u \in T_{x}M$, the curve
\begin{align}\label{eq:simple-second-order-approx-2}
	\tag{$\Ret$-2}
	\tilde x^{i}(t) := x^{i}(0) + t u^{i} - \frac 12 t^{2}\Gamma^{i}_{jk}(0)u^{j}u^{k},
	\end{align}
	provides a second-order retraction of geodesics arising from $\nabla$.
\end{proposition}
The proof immediately follows from a straightforward Taylor expansion. 
Notice that the Christoffel symbols $\Gamma^{i}_{jk}$ arising form $\nabla$ must not be confused with the ``dual'' symbols $\Gamma^{ij}_{k}$ defined in~\eqref{eq:normal-geod-Hamiltonian-2}.

\begin{remark}
	On Riemannian manifolds,~\eqref{eq:simple-second-order-approx-2} exactly corresponds to the chart-based retraction considered in~\cite{schwarz2023efficient}.
\end{remark}

\subsection{Second-order retractions for frame bundles over Riemannian manifolds}\label{ssec:numerics-horizontal-bm} Motivated by the quest for computationally efficient versions of the random walks on the frame bundle considered in Theorem~\ref{thm:conv-ret-frame-bundle}, we here provide a retraction for framed curves arising from parallel transport. For simplicity we restrict to the case where the base manifold is \emph{Riemannian}. We require the following definition:

\begin{definition}[Anisotropic frame bundle]\label{def:aniso-frame-bundle}
Let $(M,g)$ be an oriented Riemannian manifold with Levi--Civita connection $\nabla$. Denote by $\FF_{\mathrm{SO}}$ the resulting $\mathrm{SO}(n)$-subbundle of the frame bundle $\FF$. Furthermore, given a matrix $A\in \mathrm{GL}(n)$, define the anisotropic frame bundle by 
$$
\FF_{A} := \{F\cdot A  \, | \, F \in \FF_{\mathrm{SO}}\} \subset \FF \ ,
$$
using the natural right action of $\mathrm{GL}(n)$ on $\FF$. 
\end{definition}

By definition, a frame $F: \mathbb{R}^{n}\to T_{x}M$ lies in $\FF_{\mathrm{SO}}$ if $g(F(e_{i}), F(e_{j})) = \delta_{ij}$, where $(e_{1}, \dots, e_{n})$  denotes the standard basis of $\mathbb{R}^{n}$. In local coordinates this is equivalent to requiring that 
$$
F^{T}g F = \mathrm{Id} \ , 
$$ 
where one views $F$ as a square matrix, and where $g = (g_{ij})$ denotes the metric tensor in the given local chart. 

Clearly, one has $\FF_{A}  = \FF_{\mathrm{SO}}$ whenever 
$A \in \mathrm{SO}(n)$. Moreover, since parallel transport preserves $\FF_{A}$, the sub-Riemannian structure $(\HH^{\nabla}, h^{\nabla})$ defined in Definition~\ref{def:sub-Riemannian-frame-bundle} and the compatible connection $\bar \nabla$ considered in 
Definition~\ref{def:comp-conn-frame-bundle} descent to  $\FF_{A}$, modulo the obvious modifications. 

\begin{remark}
Notice that the sub-Riemannian structure on $\FF_{A}$ is bracket-generating if and only if it is bracket-generating on $\FF_{\mathrm{SO}}$.
Moreover, for the horizontal distribution on $\FF_{\mathrm{SO}}$ to be bracket-generating is a generic (i.e., an open) condition.
E.g., for surfaces this condition is equivalent to requiring that the Gau{\ss} curvature must not vanish outside a set of measure zero. For higher dimensions, a sufficient condition is provided when the Riemann curvature operator is surjective, see, e.g., Chapter 3 in~\cite{Baudoin2004}. 
\end{remark}

For the forthcoming discussion, we find it convenient to avoid index notation.

\begin{definition}[Index-free notation]\label{def:idx-free-Christoffel}
	Given a vector field $X$ on $M$ and a local chart, let the Christoffel symbols $ \Gamma^{k}_{j}(X)$ defined by  
	$$
	\nabla_{X}\partial_{j} = \Gamma^{k}_{j}(X) \partial_{k} \ ,
	$$
	where $\nabla$ denotes the Levi--Civita connection. Then let the ``index-free'' Christoffel symbol be the square matrix field
	$
	\Gamma(X) := (\Gamma^{k}_{j}(X))
	$ with row index $k$ and column index $j$. 
\end{definition}
Of course, $\Gamma^{k}_{j}(\partial_{i})=\Gamma^{k}_{ij}$ are the usual Christoffel symbols. Let $\gamma(t)$ be an integral curve of the vector field $X$, i.e., $\dot \gamma(t) = X_{|\gamma(t)}$. Then in a local chart \emph{parallel transport} of a frame $F(t): \mathbb{R}^{n}\to T_{\gamma(t)}M$ along $\gamma$ can be expressed by the matrix ODE
\begin{align}\label{eq:parallel-trasport-index-free}
\dot F + \Gamma(X)F = 0 \ .
\end{align}

We start with a second-order horizontal retraction on $\FF_{\mathrm{SO}}$: 

\begin{proposition}\label{prop:retraction-anisotropic-bm}
	Let $(M,g)$ be an oriented Riemannian manifold, and let $\gamma$ be a geodesic with initial conditions $(\gamma(0),u)$, where $u \in T_{\gamma(0)}M$. Let $F(t)$ be the $g$-orthonormal frame resulting from parallel transporting $F(0)$ from $\gamma(0)$ along $\gamma$. Let $x(t)$ denote the chart-based coordinate expression of $\gamma(t)$, and let $\Gamma(\dot x (t))$ denote the index-free Christoffel symbol from Definition~\ref{def:idx-free-Christoffel}. Define
	\begin{align}\label{eq:ret-param-frame}
	\tag{$\Ret$-3}
	\begin{split}
	\tilde x(t) &:= x(0) + t u - \frac 12 t^{2}\Gamma(u)u \ ,  \\
	E(t)&:= \left( \mathrm{Id}- t \Gamma(u) + \frac 12 t^{2}\left(\Gamma(u)^{2}- \dot \Gamma(u)\right)\right) F(0) \ , \quad \text{where} \quad \dot \Gamma(u) = \frac{d}{dt} \Gamma(\dot x(t))\lvert_{t=0} \ , \\ 
	\tilde F(t)&:= E(t) S(t)^{-1}\ , \quad \text{where} \quad S(t) := \left( E (t)^{T}\tilde g(t) E(t) \right)^{\frac 12} \ ,
	\end{split}
	\end{align}
	and where $\tilde g(t)$ denotes the metric tensor at the point $\tilde x(t)$. Then $\tilde F(t)$ is a $g$-orthonormal frame at $\tilde x(t)$. Moreover, the curve $(\tilde x (t), \tilde F(t))$ is a second-order approximation of the curve $(x (t), F(t))$ at $t=0$.
\end{proposition}

\begin{remark}
For an implementation of \eqref{eq:ret-param-frame} (and \eqref{eq:ret-param-frame-A} below), it is useful to note that 
\begin{align*}
\Gamma(u) &= u^{i}\Gamma_{ij}^{k} \, dx^{j}\otimes \partial_{k} \ , \\
\dot \Gamma(u) &= u^au^b\bigl(\Gamma_{aj,b}^k-\Gamma^i_{ab}\Gamma_{ij}^k\bigr) \, dx^{j}\otimes \partial_{k} \ ,
\end{align*}
where $\Gamma_{aj,b}^k = \partial_{b} \Gamma_{aj}^k$. The above expression for $\dot \Gamma(u)$ follows from differentiating $\Gamma(\dot x(t))$ (using Definition~\ref{def:idx-free-Christoffel}) with respect to $t$. 
These expressions apply to both the scenario where Christoffel symbols are known symbolically (i.e., arise from symbolic differentiation of the metric tensor) and to the scenario where Christoffel symbols are only available numerically. In the latter (numerical) case, one needs to compute Christoffel symbols with a second-order scheme, since this provides first-order accuracy for the spatial derivatives $\Gamma_{aj,b}^k$. Notice that this accuracy suffices for guaranteeing that \eqref{eq:ret-param-frame} remains a second-order retraction of the curve $(x(t), F(t))$.
\end{remark}	

\begin{proof}[Proof of Proposition~\ref{prop:retraction-anisotropic-bm}]
	Clearly, $\tilde x(t)$ is a second-order approximation of $x(t)$. We next show that $\tilde F(t)$ is a $g$-orthonormal frame at $\tilde x(t)$. To this end, consider the polar decomposition of $\tilde g^{\frac 12}(t) E (t)$, i.e., 
	$$
	\tilde g^{\frac 12}(t)  E (t) = R(t) S(t) \ , \quad \text{where} \quad \text{$R(t)\in O(n)$ and $S(t) = S(t)^{T}$.}
	$$
	Then $S(t)$ indeed satisfies $S(t) = \left( E (t)^{T}\tilde g(t) E(t) \right)^{\frac 12}$, and hence $\tilde F(t)= E(t) S(t)^{-1} = \tilde g^{-\frac 12}(t) R(t)$ satisfies $\tilde F(t)^{T}\tilde g(t) \tilde F(t) = \mathrm{Id}$. 
	It remains to show that $\tilde F(t)$ is a second-order approximation of the parallel frame $F(t)$. Equation~\eqref{eq:parallel-trasport-index-free} implies $\dot F(t) + \Gamma(\dot x (t))F(t) = 0$, and it follows that 
	$$
	{E(t) = F(0) + t \dot F (0) + \frac 12 t^{2}\ddot F(0) \ .}
	$$
	Therefore, $E(t)$ is a second-order approximation of $F(t)$. Let $g(t)$ denote the metric tensor at the point $x(t)$. Then $F(t)^{T} g(t) F(t) = \mathrm{Id}$. Since $\tilde x(t)$ is a second-order approximation of $x(t)$, one obtains that $g(t)- \tilde g(t) = \mathcal{O}(t^{3})$. This, together with the fact that $E(t)$ is a second-order approximation of $F(t)$, shows that $S(t) = \mathrm{Id} + {\mathcal{O}(t^{3})}$, and thus $S(t)^{-1} = \mathrm{Id} + {\mathcal{O}(t^{3})}$. Hence, $\tilde F(t)= E(t) S^{-1}(t)$ is indeed a second-order approximation of $F(t)$.
\end{proof}

A slight modification of Retraction \eqref{eq:ret-param-frame} results in a second-order retraction on $\FF_{A}$:

\begin{corollary}\label{cor:retraction-anisotropic-bm'}
    Let $(M,g)$ be an oriented Riemannian manifold, and let $\gamma$ be a geodesic with initial conditions $(\gamma(0),u)$, where $u \in T_{\gamma(0)}M$. Let $F(t)\in\mathcal F_A$ resulting from parallel transporting $F(0)\in\mathcal F_A$ from $\gamma(0)$ along $\gamma$. Let $x(t)$ denote the chart-based coordinate expression of $\gamma(t)$, and let $\Gamma(t):= \Gamma(\dot\gamma(t))$. Define
	\begin{align}\label{eq:ret-param-frame-A}
	\tag{$\Ret$-3'}
	\begin{split}
	\tilde x(t) &:= x(0) + t u - \frac 12 t^{2}\Gamma(u)u \ ,   \\
	E_A(t)&:= \left( \mathrm{Id}- t \Gamma(u) + \frac 12 t^{2}\left(\Gamma(u)^{2}- \dot \Gamma(u)\right)\right) F(0)A^{-1} \ ,\\ 
	\tilde F_A(t)&:= E_{A}(t)S_A(t)^{-1}A\ , \quad \text{where} \quad S_A(t) := \left(E_{A} (t)^{T}\tilde g(t) E_{A}(t) \right)^{\frac 12} \ ,
	\end{split}
	\end{align}
	and where $\tilde g(t)$ denotes the metric tensor at the point $\tilde x(t)$. Then $\tilde F_A(t)\in\mathcal F_A$ at $\tilde x(t)$. Moreover, the curve $(\tilde x (t), \tilde F_A(t))$ is a second-order approximation of the curve $(x (t), F(t))$ at $t=0$.
\end{corollary}
\begin{proof}
    Since $F(0)A^{-1}\in\mathcal F_{\mathrm{SO}}$, Proposition \ref{prop:retraction-anisotropic-bm} yields that $\tilde F_A(t)A^{-1}\in\mathcal F_{\mathrm{SO}}$ is a second-order approximation of $F(t)A^{-1}$. Hence, $\tilde F_{A}(t)$ is a second-order approximation of $F(t)$. 
\end{proof}	

Figure~\ref{fig:anistropic-walk-ellipsoid} provides an example of a retraction-based random walk using~\eqref{eq:ret-param-frame-A}. 
In this figure, $M$ is the parametric surface defined by $\phi(s,t)= (\cos(s) \sin(t), \sin(s) \sin(t), 1.5\cos(t))$, equipped with the induced Riemannian metric from ambient $\mathbb{R}^{3}$. Using an implementation in \textsc{Mathematica} on a standard laptop, the computation takes a few seconds for $20.000$ steps and stepsize $\varepsilon=0.05$.

	\printbibliography
	
\end{document}